\documentclass[12pt]{amsart}

\usepackage{amssymb}
\usepackage{xypic}
\usepackage{tikz}

\theoremstyle{remark}
\newtheorem{example}{Example}[section]

\theoremstyle{plain}
\newtheorem{proposition}[example]{Proposition}

\newtheorem{corollary}[example]{Corollary}
\newtheorem{theorem}[example]{Theorem}

\newcommand{\ZZ}{\ensuremath{{\mathbb Z}}}
\newcommand{\QQ}{\ensuremath{\mathbb Q}}
\newcommand{\CC}{\ensuremath{\mathbb C}}

\newcommand{\PP}{\ensuremath{\mathbb P}}
\newcommand{\gm}{{\mathbb G}_m}

\newcommand{\sO}{{\mathcal O}}
\newcommand{\sF}{{\mathcal F}}
\newcommand{\sE}{{\mathcal E}}
\newcommand{\sL}{{\mathcal L}}
\newcommand{\sG}{{\mathcal G}}
\newcommand{\sK}{{\mathcal K}}

\newcommand{\sN}{{\mathcal N}}

\newcommand{\stX}{{\mathfrak X}}

\newcommand{\coh}[2]{{\rm H}^{#1}(#2)}

\usetikzlibrary{matrix,arrows}

\newcommand{\sV}{{\mathcal V}}

\renewcommand{\AA}{\mathbb A}

\newcommand{\et}{{\rm et}}
\newcommand{\vect}{{\rm Vect}}
\newcommand{\op}{{\rm op}}
\newcommand{\red}{{\rm red}}

\newcommand{\pardeg}{{\rm deg}_{\rm par}}
\newcommand{\rk}{{\rm rk}}

\renewcommand{\c}{{\rm c}^\et}
\newcommand{\ch}{{\rm ch}^\et}
\newcommand{\td}{{\rm td}^\et}

\newcommand{\chrep}{{\rm ch}^{\rm rep}}
\newcommand{\tdrep}{{\rm td}^{\rm rep}}

\newcommand{\sT}{{\mathcal T}}

\begin{document}

\title[Semistability criterion for parabolic bundles]{Semistability
criterion for parabolic vector bundles on curves}

\author[I. Biswas]{Indranil Biswas}

\address{School of Mathematics, Tata Institute of Fundamental
Research, Homi Bhabha Road, Bombay 400005, India}

\email{indranil@math.tifr.res.in}

\author[A. Dhillon]{Ajneet Dhillon}

\address{Department of Mathematics, University of Western Ontario,
London, Ontario N6A 5B7, Canada}

\email{adhill3@uwo.ca}

\subjclass[2000]{14F05, 14H60}

\keywords{Parabolic bundle, root stack, semistability, cohomology}

\date{}

\begin{abstract}
We give a cohomological criterion for a parabolic vector bundle on
a curve to be semistable. It says that a parabolic vector bundle
$\sE_*$ with rational parabolic weights is semistable if and only 
if there is another parabolic
vector bundle $\sF_*$ with rational parabolic 
weights such that the cohomologies of the vector bundle 
underlying the parabolic tensor product $\sE_*\otimes\sF_*$ 
vanish. This criterion generalizes the known semistability 
criterion of Faltings for vector bundles on curves and
significantly improves the result in \cite{biswas:07}.
\end{abstract}

\maketitle

\section{Introduction}

We will work over an algebraically closed ground field of 
characteristic zero.

Let $X$ be an irreducible smooth projective curve. A theorem due to 
Faltings says that a vector bundle $E$ over $X$ is semistable if and
only if there is a vector bundle $F$ over $X$ such that
$\coh{0}{X,\, E\otimes F}\,=\,0\,=\, \coh{1}{X,\, E\otimes F}$
(see \cite[p. 514, Theorem 1.2]{faltings:94} and \cite[p. 516, 
Remark]{faltings:94}). Let $D$ be a reduced effective 
divisor on $X$.
For a parabolic vector bundle $W_*$ on $X$ with 
parabolic divisor $D$, the underlying vector bundle will be 
denoted by $W_0$; see \cite{MS}, \cite{MY} for parabolic vector
bundles. Let $r$ be
a positive integer. Denote by $\vect(X,D,r)$ the category of parabolic
vector bundles on $X$ with parabolic
structure along $D$ and parabolic weights being integral multiples
of $1/r$. In \cite{biswas:07} the following theorem was proved:

\begin{theorem}\label{th1}
There is a parabolic vector bundle $\sV_*\,\in\,\vect(X,D,r)$ with the
following property: A parabolic vector bundle $\sE_*$ is semistable
if and only if there is a parabolic vector bundle 
$\sF_*\in\vect(X,D,r)$
with $\coh{0}{X,\, (\sE_*\otimes\sV_*\otimes\sF_*)_0}\,=\,0\,=\,
\coh{1}{X,\, (\sE_*\otimes\sV_*\otimes\sF_*)_0}$, where
$(\sE_*\otimes\sV_*\otimes\sF_*)_*$ is the parabolic tensor product.
\end{theorem}

Theorem \ref{th1} was also proved in \cite{Pa}.
It should be mentioned that the vector bundle $\sV_*$ in Theorem 
\ref{th1} is not canonical; it depends upon the choice of a 
suitable ramified Galois covering $Y\longrightarrow X$
that transforms parabolic bundles in $\vect(X,D,r)$ into 
$G$-linearized vector bundles on $Y$, where $G$ is the Galois group
for the covering. However, many different covers do this.

We prove that $\sV_*$ in Theorem \ref{th1} can be chosen to be
the trivial line bundle ${\mathcal O}_X$ equipped with the trivial
parabolic structure. More precisely, we
prove the following theorem (see Theorem \ref{t:faltings}):

\begin{theorem}\label{th2}
A parabolic vector bundle $\sE_*\,\in\,\vect(X,D,r)$
is semistable if and only if there is a parabolic vector bundle
$\sF_*\,\in\,\vect(X,D,r)$
such that
$$\coh{0}{X,\, (\sE_*\otimes\sF_*)_0}\,=\,0\,=\, 
\coh{1}{X,\, (\sE_*\otimes\sF_*)_0}\, .$$
\end{theorem}

Theorem \ref{th2} is proved by systematically working with stacks.
Compare this method with the earlier attempts (cf. \cite{biswas:07},
\cite{Pa}) that landed in the weaker version given in Theorem \ref{th1}.

\section{Parabolic bundles and root stacks}

Recall that to give a morphism $X\,\longrightarrow\, [\AA^1/\gm]$
is the same as giving a line bundle $\sL$ with section $s$
on $X$ (see \cite{cadman:06}). Given a positive
integer $r$, there is a natural morphism
$$
\theta_r \,:\, [\AA^1/\gm] \,\longrightarrow \,[\AA^1/\gm]
$$
defined by $t\, \longmapsto\, t^r$, with $t\, \in\, \AA^1$.
We define the root stack 
$X_{(\sL,s,r)}$ to be the fibered product
$$
X\times_{[\AA^1/\gm],\theta_r} [\AA^1/\gm]\, .
$$
When the section is non-zero, this root stack is an orbifold
curve; see \cite[Example 2.4.6]{cadman:06}. 

The data $(\sL,s)$ corresponds to an effective divisor $D$ on 
$X$. We will
henceforth assume that this divisor is reduced. Sometime we write
$X_{D,r}$ instead of $X_{\sL,s,r}$.

We think of the ordered set $\frac{1}{r}\ZZ$ of rational numbers with 
denominator $r$ as a category. 
Let $j$ be an integer multiple of $1/r$. Given a functor from the
opposite category
$$
\sF_* \,:\, (\frac{1}{r}\ZZ)^\op\,\longrightarrow\, \vect(X)\, ,
$$
we denote by $\sF_*[j]$ its shift by $j$, so
$$
\sF_i[j] \,=\,\sF_{i+j}\, .
$$
There is a natural transformation $\sF_*[j] \,\longrightarrow \,
\sF_*$ when $j\,\ge\, 0$.

A \emph{vector bundle with parabolic structure} over $D$ such that
the parabolic weights are integral multiples of $1/r$ is a functor
$$
\sF_* \,:\, (\frac{1}{r}\ZZ)^\op \,\longrightarrow\, \vect(X)
$$
together with a natural isomorphism
$$
j\,:\, \sF_*\otimes \sO_X(-D) \,\stackrel{\sim}{\longrightarrow}\,
\sF[1]
$$
such that the following diagram commutes
$$
\xymatrix{
 \sF_*\otimes \sO_X(-D) \ar[r] \ar[dr] & \sF[1] \ar[d] \\
 & \sF_*
}
$$
(see \cite{MY}, \cite{MS}).
The \emph{underlying vector bundle} of a parabolic vector bundle is the 
value of this
functor at $0$. We have previously denoted this by $\sF_0$. For a 
functor $\sF_*$ defining a parabolic vector bundle, the value of
$\sF_*$ at $t\, \in\, \frac{1}{r}\ZZ$ will be denoted by
$\sF_t$.

Denote by $\vect(X,D,r)$ the category of vector bundles on $X$ 
with parabolic structure along $D$ and parabolic weights integral 
multiples of $1/r$. It is a tensor category.

\begin{theorem}
\label{t:equiv}
There is an equivalence of tensor categories
$$
F\,:\, \vect(X_{(\sL,s,r)}) \,\stackrel{\sim}{\longrightarrow}\, 
\vect(X,D,r).
$$
The equivalence preserves parabolic degree and semistability (see \S~ 
\ref{s:semi} below).
\end{theorem}

The functor $F$ has the following explicit description. There is 
a natural root line bundle $\sN$ on $X_{(\sL,s,r)}$. Given a 
vector bundle $\sF$ on the root stack, the corresponding
parabolic bundle is the functor defined by
$$
l/r \,\longmapsto\, \pi_*(\sF\otimes \sN^l)\, .
$$

\begin{proof}[Proof of Theorem \ref{t:equiv}]
See \cite[Section 3]{borne:07} and \cite{biswas:97}.
\end{proof}

\section{Root stacks as quotient stacks}

For the map $z\, \longmapsto\, z^n$ defined around $0\, \in\,
\mathbb C$, the \textit{ramification index} at $0$ will be $n-1$.

We will need the following theorem :

\begin{theorem}\label{th-na}
Suppose $k=\CC$.
There is a finite Galois covering $Y\longrightarrow X$ ramified 
over $D$ 
with ramification index $r-1$ at each point in $D$ if and only if
either $X\,\ne\, \PP^1$ or $X\,=\,\PP^1$ with $|D|\,\not=\,1$. 
\end{theorem}

\begin{proof}
See \cite[p. 29, Theorem 1.2.15]{namba:87}.
\end{proof}

\begin{corollary}\label{cor}
Theorem \ref{th-na} holds over any algebraically closed ground field of 
characteristic zero.
\end{corollary}

\begin{proof}
This follows from \cite[Expose IX, Theorem 4.10]{sga1}.
See also Proposition 7.2.2 in \cite[p. 146]{murre:67}. 
\end{proof}

\begin{proposition} \label{p:quot}
Suppose that either $X\ne \PP^1$ or $|D|\not= 1$. Then $X_{(D,r)}$ is a 
quotient stack.
\end{proposition}

\begin{proof}
Fix a covering $Y\longrightarrow X$ as in Corollary \ref{cor}. Let
$G$ be the Galois group for this covering. Our goal 
is to show that $X_{(D,r)} \,=\, [Y/G]$.

Let $R$ be the ramification divisor in $Y$. Then the reduced 
divisor $R_\red$ produces a morphism
\begin{equation}\label{177}
Y\,\longrightarrow \,X_{(D,r)}
\end{equation}
via the universal property of root stacks. As $R_\red$ is
$G$-invariant so is the morphism in \eqref{177}. Hence we
obtain a morphism
$$
[Y/G]\,\longrightarrow \,X_{(D,r)}\, .
$$
To show that this morphism is an isomorphism is a local condition
for the flat topology and follows from \cite[Example 2.4.6]{cadman:06}. 
\end{proof}

\section{Semistability}\label{s:semi}

Recall that the \emph{parabolic degree} of a parabolic
vector bundle $\sE_*$ over $X$ is defined to be
$$
\begin{array}{ccc}
\pardeg(\sE_*) & := & 
\rk(\sE_0)(\deg D - \chi({\mathcal O}_X)) + 
\frac{1}{r}(\sum_{i=1}^{r}\chi(\sE_{i/r})) \\
 & = & \rk(\sE_0)\deg D + \frac{1}{r} \sum_{i=1}^{r} \deg(\sE_{i/r})
\end{array}
$$
(see \cite{MS}, \cite{biswas:97}, \cite[\S~4]{borne:07}).
The \textit{slope} is defined as usual :
$$
\mu(\sE_*) \, :=\, \frac{\pardeg(\sE)}{\rk(\sE)}\, .
$$
A parabolic vector bundle $\sE_*$ is said to be \emph{semistable} if
$$
\mu(\sE_*) \,\ge\, \mu(\sF_*)
$$
for all parabolic subbundles $\sF_*$.

\begin{example} \label{e:semi}
Let us describe all the parabolic semistable bundles on $\PP^1$ with one 
parabolic point, meaning
$D\,=\,x$, where $x$ is some point on $\PP^1$. Let $\sE_*$ be a 
semistable parabolic vector bundle. Then we may write
$$
\sE_0 = \bigoplus_{k=1}^m \sO(n_k)^{s_k}
$$
\cite{Gr}. We may assume that the integers $n_i$ are strictly 
decreasing. A subbundle $\sF_*$ is defined by taking
$$
\sF_{i/r} = \sO(n_1)^{s_1} \cap \sE_{i/r}
$$
for $0\le i < r$. This extends to a parabolic subbundle of $\sE_*$. We
see immediately that
$$
\mu(\sF_*) \,>\, \mu(\sE_*)
$$
when $m>1$. Consequently, a parabolic vector bundle $\sE_*$ of rank
$n$ over $\PP^1$ with one parabolic point is semistable if and only if
$$
\sE_*\, =\, ({\mathcal L}_*)^{\oplus n}\, ,
$$
where ${\mathcal L}_*$ is a parabolic line bundle.
\end{example}

\section{Grothendieck-Riemann-Roch theorem for Deligne-Mumford stacks}

In this section we recall the pertinent results from 
\cite{toen:99}. An excellent summary
of this paper of T\"oen can be found in the appendix to \cite{borne:07}. 
We denote
by $\stX$ a smooth Deligne-Mumford stack that is proper over our ground 
field $k$. We equip it with the \'etale topology. 
The category of vector bundles (respectively, coherent sheaves) on 
$\stX$ is an exact category so 
we may form the groups
$$
K_i(\stX) \qquad (\text{respectively, } G_i(\stX) )\, .
$$

Let $\sK_i$ denote the sheaf in the \'{e}tale topology on $\stX$ 
associated to the presheaf
$$
(X \longrightarrow \stX) \,\longmapsto \,K_i(X)\, .
$$
Set
$$
\coh{i}{\stX, \,\QQ} \,=\, \coh{i}{\stX,\, \sK_i\otimes\QQ}\, .
$$

By \cite{gillet:81} we have Chern classes and hence
Chern characters and Todd classes
$$
\c_i,\ \ch,\ \td : K_0(\stX)\longrightarrow \coh{*}{\stX}\, .
$$

Let $I_\stX \,:=\, \stX \times_{\stX\times\stX} \stX$ be the inertia 
stack of $\stX$. Let $\mu_\infty$ denote the group of roots of unity in 
$\overline{\QQ}$, and set $\Lambda\,:=\, \QQ(\mu_\infty)$. If $\sG$ 
is a locally free sheaf on $I_\stX$, the inertial action induces an 
eigenspace decomposition 
$$
\sG \,=\, \bigoplus_{\zeta\in\mu_\infty} \sG^{(\zeta)}\, .
$$
Let
$$
\rho_\stX \,:\, K_0(I_\stX)\otimes_\ZZ\Lambda \,\longrightarrow\,
K_0(I_\stX)\otimes_\ZZ\Lambda
$$
be the morphism defined by
$$
\sG \longmapsto \sum \zeta [\sG^{(\zeta)}]\,.
$$
We have a morphism, called the \emph{Frobenius character},
$$
\phi_\stX\,:\, K_0(\stX)\otimes_\ZZ\Lambda \,
\stackrel{\pi^*_\stX}{\longrightarrow}\,
K_0(I_\stX)\otimes_\ZZ\Lambda\,\stackrel{\rho_\stX}{\longrightarrow}
\,K_0(I_\stX)\otimes_\ZZ\Lambda\,\longrightarrow\,
K_{0,\et}(I_\stX)\otimes_\ZZ\Lambda\, .
$$

The ring $K_0$ is a lambda ring and we write $\lambda_{-1}(x) = \sum 
(-1)^i \lambda_i(x)$. Define
$$
\alpha_\stX \,:=\, 
\rho_\stX(\lambda_{-1}([\Omega^1_{I_\stX/\stX}]))\, \in 
\, K_{0,\et}(I_\stX)\otimes_\ZZ \Lambda\, .
$$

Finally define the characteristic classes 
$$
\chrep(x)\, :=\, \ch(\phi_\stX(x))
$$
and
$$
\tdrep(\stX) \,:=\, \ch(\alpha_\stX^{-1})\td(\sT_{I_\stX})\, .
$$

\begin{theorem}\label{thb}
Denote by $\int^{\rm rep}_\stX$ the push-forward $p_*$ for $p\,:\,
I_\stX \,\longrightarrow \,{\rm Spec}(k)$. The following holds:
$$
\chi(\stX,\sF) \,=\, \int_\stX^{\rm rep} \tdrep(\stX)\chrep(\sF)\, .
$$
\end{theorem}

\begin{proof}
See \cite[Corollary 4.13]{toen:99}.
\end{proof}

\begin{corollary}\label{cori}
Suppose that $\stX$ is a proper orbifold curve. Then
$$
\mu(\sF)\,=\,\chi(\sF) -\int^{\rm rep}_\stX \tdrep(\stX)\, .
$$
\end{corollary}

\begin{proof}
We have that $\pi_\stX^*(\sF)$ is an eigensheaf with eigenvector 
$1$ as the
stack $\stX$ is generically a variety. There is a diagram
$$
\xymatrix{
 I_\stX \ar[dr]^{p_I} \ar[d]^{\pi_\stX}& \\
 \stX \ar[r]^(0.3){p} & {\rm Spec}(k). \\
}
$$
By the projection formula,
$$
p_{I,*}(c_1^\et(\pi^*_{\stX}\sF)) \,=\, p_*(c_1^\et(\sF))\, .
$$
In view of Theorem \ref{thb}, the result follows from the fact that
$\deg(\sF)) \,=\, p_* (c_1^\et(\sF))$ (\cite[Theorem 
4.3]{borne:07}) and the usual expression for the Chern character.
\end{proof}

\begin{corollary} 
\label{c:semi}
Suppose that there is a vector bundle $\sE$ so that
$\coh{i}{\stX,\,\sE\otimes \sF}\,=\,0$ for $i\,=\,
0\, ,1$. Then $\sF$ is semistable.
\end{corollary}

\begin{proof}
Suppose there is a subsheaf $\sF'$ of $\sF$ with 
$$
\mu(\sF') \,>\, \mu(\sF)\, .
$$
Then it follows from Corollary \ref{cori} that
$$
\frac{\chi(\sE\otimes \sF')}{\text{rank}(\sE\otimes\sF')}
-\frac{\chi(\sE\otimes \sF)}{\text{rank}(\sE\otimes\sF)}\, >\, 0\, .
$$
Since $\chi(\sE\otimes \sF)\,=\, 0$, this implies that
$\coh{0}{\stX,\, \sE\otimes \sF'}\,\not=\,0$. But
$\sE\otimes \sF'\, \subset\, \sE\otimes \sF$. Hence
$\coh{0}{\stX,\, \sE\otimes \sF}\,\not=\,0$
which is a contradiction
\end{proof}

\section{Semistability criterion}

\begin{theorem} \label{t:faltings}
A vector bundle with parabolic structure $\sE_*\,\in\,\vect(X,D,r)$
is semistable if and only if there is a parabolic vector bundle
$\sF_*\in\vect(X,D,r)$ with
$$\coh{i}{X,\,(\sE_*\otimes\sF_*)_0}\,=\,0$$ for all $i$, where
$(\sE_*\otimes\sF_*)_*$ is the parabolic tensor product.
\end{theorem}

\begin{proof}
We have a morphism $\pi: X_{D,r} \longrightarrow X$, and $\pi_*$ 
is exact as ${\rm char}(k)\,=\,0$. 
Hence by the Leray spectral sequence,
$$
\coh{i}{X,\,\pi_*(\sF)} \,=\, \coh{i}{X_{D,r},\,\sF}
$$
for all $i$.

Suppose that
there is a parabolic vector bundle $\sF_*\,\in\,\vect(X,D,r)$
with 
$$\coh{0}{X,\,(\sE_*\otimes\sF_*)_0}\,=\,0\, =\, 
\coh{1}{X,\,(\sE_*\otimes\sF_*)_0}\, .$$
Applying Theorem \ref{t:equiv}, we deduce from
Corollary \ref{c:semi} that $\sE_*$ is semistable.

To prove the converse, assume that $\sE_*$ is semistable.
We break up into two cases.

The case of $\PP^1$ with exactly one parabolic point: Applying 
Example \ref{e:semi}, we see that
$$
\sE_0 \,= \,\bigoplus \sO(n)^m\, .
$$
So tensoring with $\sO(-n-1)$ does the job.

All other cases: In view of Proposition \ref{p:quot} we may 
assume that we 
have a quotient stack, so $X_{D,r} \,=\, [Y/G]$.
Then given a semistable parabolic bundle on $X$, we obtain a 
corresponding semistable $G$-linearized vector bundle $\sE$ on $Y$.
We note that this implies that the vector bundle $\sE$ is
semistable \cite[p. 308, Lemma 2.7]{biswas:97}.
By \cite[p. 514, Theorem 1.2]{faltings:94}, there is a vector bundle 
$\sF$ on $Y$ such that all
the cohomology groups of $\sF\otimes \sE$ vanish. Consider
$$
\widetilde{\sF} \,= \,\bigoplus_{g\in G} g^*\sF\, .
$$
The vector bundle $\widetilde{\sF}$ has a natural $G$-action and 
$$
\coh{i}{Y,\,\widetilde{\sF}\otimes\sE} \,=\, 0
$$
for all $i$. The vector bundle $\widetilde{\sF}$ produces a 
vector bundle on $[Y/G]$,
which will also be denoted by $\widetilde{\sF}$. Finally
$$
\coh{i}{[Y/G],\, \widetilde{\sF}\otimes\sE} \,=\, \coh{i}{Y, 
\,\widetilde{\sF}\otimes\sE}^G \,= \,0\, .
$$
The theorem now follows.
\end{proof}

\medskip
\noindent
\textbf{Acknowledgements.}\, We thank the Kerala School of Mathematics,
where a part of the work was carried out, for its hospitality.

\end{document}